\documentclass[english,a4paper,12pt,reqno]{article}

\usepackage{amsmath,amsfonts,latexsym,amssymb,amsthm}
\usepackage[tips,matrix,arrow,frame]{xy}

\usepackage{xcolor}

\newcommand{\C}{\mathbb{C}}

\newcommand{\Z}{\mathbb{Z}}

\newcommand{\fgl}{\mathfrak{gl}}

\newcommand{\fsp}{\mathfrak{sp}}

\newcommand{\Ad}{\mathrm{Ad}}
\newcommand{\End}{\mathrm{End}}

\newcommand{\Ind}{\mathrm{Ind}}

\newcommand{\Prim}{\mathrm{Prim}}

\newcommand{\nod}{{}_{\circ}^{\circ}}
\newcommand{\sgn}{\mathrm{sgn}}
\newcommand{\Tr}{\mathrm{Tr}}
\newcommand{\triv}{\mathrm{triv}}

\newcommand{\vacl}{\langle \mathrm{vac} \vert}
\newcommand{\vacr}{\vert \mathrm{vac}\rangle}

\newtheorem{thm}{Theorem}[section]
\newtheorem{prop}[thm]{Proposition}
\newtheorem{lemma}[thm]{Lemma}
\newtheorem{cor}[thm]{Corollary}

\newtheorem{rem}[thm]{Remark}
\newtheorem{ack}{Acknowledgment}

\newcommand{\fg}{\mathfrak{g}}

\newcommand{\fo}{\mathfrak{o}}

\newcommand{\fS}{\mathfrak{S}}

\newcommand{\an}{\mathrm{ann}}
\newcommand{\cre}{\mathrm{cr}}
\newcommand{\even}{\mathrm{even}}
\newcommand{\odd}{\mathrm{odd}}

\newcommand{\cA}{\mathcal{A}}

\newcommand{\cF}{\mathcal{F}}

\newcommand{\cM}{\mathcal{M}}

\newcommand{\cR}{\mathcal{R}}

\newcommand{\cT}{\mathcal{T}}

\newcommand{\cW}{\mathcal{W}}

\newcommand{\bD}{\mathbf{D}}

\begin{document}
\title{Homology of Lie algebras of \\Orthogonal and Symplectic generalized Jacobi matrices}
\author{A. Fialowski$^{a}$ and K. Iohara$^{b}$}

\maketitle

\begin{center}\begin{flushleft}
{\small{ {}$^b$  University of P\'ecs and E\"{o}tv\"{o}s Lor\'{a}nd University 
                       Budapest,  Hungary\\
  \qquad e-mail : fialowsk@ttk.pte.hu, fialowsk@cs.elte.hu \\
          {}$^a$ Universit\'{e} de Lyon, Universit\'{e} Lyon 1,
   CNRS UMR 5208, Institut Camille Jordan,
   43 Boulevard du 11 Novembre 1918,
   F-69622 Villeurbanne Cedex, France. \\
  \qquad e-mail : iohara@math.univ-lyon1.fr\; (corresponding author)\\
        }}
\end{flushleft}\end{center}

\begin{abstract}  In this note, we compute the homology with trivial coefficients of Lie algebras of generalized Jacobi matrices of type $B, C$ and $D$ over an associative unital $k$-algebra with $k$ being a field of characteristic $0$. 
\end{abstract}



\section{Introduction}
In the first half of 1980's, a certain version of the infinite rank Lie algebra $\fgl(\infty)$ and some of their subalgebras have been extensively used to describe the soliton solutions of the Kadomtsev-Petviashvili (KP in short) like hierarchies (see, e.g., \cite{JM} for detail). Because of their importance, these algebras are still studied both in mathematics and physics. Nevertheless, their basic algebraic properties are not well-understood.

The first result about the homology with trivial coefficients of the Lie algebra $\fgl(\infty)$ over a field $k$ of characteristic $0$ appeared in the short note of B. Feigin and B. Tsygan \cite{FT} in 1983. Unfortunately, their short note was too dense and the proofs were not precise. Nevertheless, their results turned out to be correct as we showed with detailed proof in our previous work \cite{FI}. We also managed to generalize it to the coefficients over an associative unital $k$-algebra $R$. 

In this article, we obtain analogous descriptions of the trivial homology spaces
for orthogonal and symplectic subalgebras of $\fgl(\infty)$ over $R$. Such results seem not to be known even for the simplest case $R=k$.
For the Lie algebra $\fgl(\infty)$, the main tools were the Loday-Quillen \cite{LQ} and Tsygan \cite{T} theorem on the stable homology of $\fgl$ and cyclic homology computations. For the orthogonal and symplectic subalgebras, we need different tools: the Loday-Procesi theorem \cite{LP} on the stable homology of $\fo$ and $\fsp$ and dihedral homology \cite{L}. As in the case of $\fgl(\infty)$, in this work we also observe \textbf{delooping} phenomenon, that is, a degree shift of the homologies of $\fo$ (resp. $\fsp$) and orthogonal (resp. symplectic) subalgebras of $\fgl(\infty)$.

The first crucial step is to show that we can apply Loday-Procesi's result \cite{LP} on expressing the primitive part of the homology of our Lie algebras in terms of the skew-dihedral homologies of the $k$-algebra $J(R)$ of generalized Jacobi matrices. The second important step is to relate the skew-dihedral homology of $J(R)$ to the dihedral homology of $R$ via an explicit isomorphism between the Hochschild homologies of $J(R)$ and of $R$, given in \cite{FI}.

The paper is organized as follows. In Section 2, we recall the definition of the $k$-algebra $J(R)$ of generalized Jacobi matrices, the Lie algebras of type $A_J, B_J, C_J$ and $D_J$ and review Lie algebra homology and (skew-)dihedral homology. 
Denote the Lie algebra structure of $J(R)$ by $\fg J(R)$.
Section 3 is a key step for the results,  namely, we consider the restrictions of the Lie algebra isomorphisms $\fg J(R) \overset{\sim}{\rightarrow} \fgl_n(J(R))$ ($n \in \Z_{>1}$), considered by B. Feigin and B. Tsygan \cite{FT}, to the Lie subalgebras of type $B_J, C_J$ and $D_J$. As a consequence, the homology of these Lie algebras can be expressed as the stable limit of the corresponding finite dimensional classical simple Lie algebras which allows us to apply the Loday-Procesi theorem on the primitive part of our homology. Finally, in Section 4, we compute the homology of the Lie algebras $\fg$ over $R$ of type $B_J, C_J$ and $D_J$. As a corollary, we obtain the universal central extensions of our subalgebras.
The main result says that the primitive part of all these Lie algebras is the dihedral homology of $R$ shifted by $2$:  
\[ \Prim(H_\bullet(\fg(R)))=HD_{\bullet-2}(R). \]
In Appendices A and B we present some backgrounds from soliton theory for $R=\C$.
In Section A, we explain how the so-called Japanese cocycle was discovered. In 
Section B, we recall the Lie algebras of orthogonal and symplectic subalgebras of $\fg J(\C)$ with their central extensions which were given in the same way as for $\fg J(\C)$ explained in Section A. Our result confirms that the central extension obtained this way  is universal. 

\begin{ack}
The authors would like to thank Claudio Procesi for useful discussions and the referee for 
careful reading. 
\end{ack}

\section{Preliminaries}

\subsection{Generalized Jacobi Matrices}
Let $k$ be a field of characteristic $0$ and let $R$ be an associative unital $k$-algebra. 
As an $R$-module, $J(R)$ is spanned by matrices indexed over $\Z$:
\[ J(R)=\{ \,(m_{i,j})_{i,j \in \Z}\, \vert \,m_{i,j} \in R, \; \exists\,\, N \,\, \text{s.t.} \,\, m_{i,j}=0 \quad (\text{$\forall\, i,j$} \,\, \text{s.t.} \,\, \vert i-j\vert>N)\, \}.
\]
With the standard operations on matrices, $J(R)$ has a structure of associative algebra. The usual Lie bracket $[A,B]:=AB-BA$ on $J(R)$ is well-defined, and we shall denote it by $\fg J(R)$, whenever we regard it as a Lie algebra.   
Denote the matrix elements of $J(R)$ by $E_{i,j}$\, ($i,j \in \Z$). For $r \in R$, we set $E_{i,j}(r)=rE_{i,j}$.\\

Assume that $R$ is equipped with a $k$-linear anti-involution $\overline{\cdot}:R \rightarrow R$ and the transpose map ${}^t(\cdot):J(R) \rightarrow J(R)$ which are defined in \S \ref{sect_involution-I} and \S \ref{sect_involution-II}. We set $R_{\pm 1}=\{r \in R \vert \bar{r}=\pm r\}$. It is clear that $R=R_1 \oplus R_{-1}; r \mapsto \frac{1}{2}(r+\bar{r})+\frac{1}{2}(r-\bar{r})$. \\
Now we extend the definition of the \textbf{transpose}, also denoted by ${}^t(\cdot )$, to $J(R)$ by ${}^t(E_{i,j}(r))=E_{j,i}(\overline{r})$.

For $l \in \Z$, let $\tau_l, \tau_l^s$ be the $k$-linear anti-involutions of the Lie algebra $\fg J(R)$ defined by 
\begin{align*}
& \tau_l(X)=(-1)^lJ_l\cdot {}^t(X)\cdot J_l, \qquad J_l=\sum_{i \in \Z} (-1)^{i}E_{i,l-i}, \\
& \tau_l^s(X)=J_l^s\cdot {}^t(X)\cdot J_l^s, \qquad J_l^s=\sum_{i \in \Z} E_{i,l-i}. 
\end{align*}
See also Appendix \ref{sect_involution-I} for the anti-involutions $\tau_l$. 
\begin{rem} \label{rem_matrix-J}
The matrices $J_l$ and $J_l^s$ do not belong to $J(R)$, \\
but the multiplication of such matrices with any matrix from $M_\Z(R)=\{(m_{i,j})_{i,j \in \Z} \vert m_{i,j} \in R\}$ is well-defined, and the map $X \mapsto J_l X J_l$ is a well-defined $k$-endomorphism of $J(R)$.
\end{rem}
\subsection{Lie Algebras of type $A_J, B_J, C_J$ and $D_J$}
We say that the universal central extension of the Lie algebra $\fg J(R)$ is of type $A_J$. The Lie algebras of other types $B_J, C_J$ and $D_J$ are defined as their fix point subalgebras with respect to appropriate involutions. 
More precisely, 
the Lie algebras over $R$ of type $B_J, C_J$ and $D_J$ are universal central extensions of $\fo_J^{\odd}(R), \fsp_J(R)$ and $\fo_J^{\even}(R)$  
\begin{align*} 
\fo_J^{\odd}(R):=
&\fg J(R)^{\tau_0^s,-}=\{X \in \fg J(R) \vert \tau_0^s(X)=-X\}, \\
\fsp_J(R):=
&\fg J(R)^{\tau_{-1},-}=\{X \in \fg J(R) \vert \tau_{-1}(X)=-X\}, \\
\fo_J^{\even}(R):=
&\fg J(R)^{\tau_{-1}^s,-}=\{X \in \fg J(R) \vert \tau_{-1}^s(X)=-X\}.
\end{align*}

\begin{rem} In \cite{JM} ,
the universal central extensions of the Lie algebras $\fo_J^{\odd}(\C), \fsp_J(\C)$ and $\fo_J^{\even}(\C)$ are called of type $B_\infty, C_\infty$ and $D_\infty$, respectively,  where they used these Lie algebras to obtain Hirota bilinear forms with these symmetries. See Appendix \ref{sect_BCD-JM} for some explanation. 
\end{rem}
We avoid using the notations $A_\infty, B_\infty, C_\infty$ and $D_\infty$ since they are also used to indicate the stable limit of the corresponding finite dimensional simple Lie algebras. \\

Let us write explicitly the direct summands, as vector spaces, of the Lie algebras $\fo_J^{\odd}(R), \fsp_J(R)$ and $\fo_J^{\even}(R)$ for later use: \\
\vskip+0.1in

\noindent \fbox{$\fo_J^{\odd}(R)$} \hspace{0.2in}
$R_1(E_{r,s}-E_{-s,-r})\oplus R_{-1}(E_{r,s}+E_{-s,-r})$ \\
\phantom{ABCDEFGh} is a direct summand of $\fo_J^{\odd}(R)$ for any $r,s\in \Z$. \\

\vskip+0.15in

\noindent \fbox{$\fsp_J(R)$} \hspace{0.2in}
$R_1(E_{r,s}-(-1)^{r+s}E_{-s-1,-r-1})\oplus R_{-1}(E_{r,s}+(-1)^{r+s}E_{-s-1,-r-1})$ \\
\phantom{ABCDEFGh} is a direct summand of $\fsp_J(R)$ for any $r,s \in \Z$. \\

\vskip+0.15in

\noindent \fbox{$\fo_J^{\even}(R)$} \hspace{0.2in}
$R_1(E_{r,s}-E_{-s-1,-r-1})\oplus R_{-1}(E_{r,s}+E_{-s-1,-r-1})$ \\
\phantom{ABCDEFGh} is a direct summand of $\fo_J^{even}(R)$ for any $r,s \in \Z$. \\

\vskip+0.15in

\noindent Notice that one of the above components can be the zero subspace.

\subsection{Lie Algebra Homology}\label{sect_Lie-homology}
Our goal is to compute the homology of Lie algebras $\fo_J^{\odd}(R), \fsp_J(R)$ and $\fo_J^{\even}(R)$ with coefficients in the trivial module $k$. For the sake of brevity, let us denote these Lie algebras by $\fg$. The homology groups are the homology of the complex $(\bigwedge^{\bullet}\fg, d)$, called the \emph{Chevalley-Eilenberg complex}, where $\bigwedge^{\bullet}\fg$ is the exterior algebra and the differential $d$ is given by
\[
d(x_1\wedge \cdots \wedge x_n) 
:=
\sum_{1\leq i<j\leq n}(-1)^{i+j+1}[x_i,x_j]\wedge x_1 \wedge \cdots \wedge
\hat{x_i} \wedge \cdots \wedge \hat{x_j} \wedge \cdots \wedge x_n.
\]
The homology $H_\bullet(\fg)$ has a commutative and cocommutative DG-Hopf algebra structure (cf. \cite{Q}). Hence, the Quillen version of the Milnor-Moore theorem \cite{MM} states that it is the graded symmetric algebra over its primitive part. 
\subsection{(Skew-)Dihedral Homology}\label{sect_dihedral}
Consider the $n+1$ ($n \in \Z_{\geq 0}$) tensor product of an associative unital $k$-algebra $\cR$ with an anti-involution $\overline{\cdot}:\cR \rightarrow \cR$. The dihedral group $D_{n+1}:=\langle x,y \vert x^{n+1}=y^2=1,\, yxy=x^{-1}\rangle$
acts on $\cR^{\otimes n+1}$ by 
\begin{align*}
&x.(r_0\otimes r_1\otimes \cdots \otimes r_n)=(-1)^nr_n\otimes r_0\otimes r_1 \otimes \cdots \otimes r_{n-1}, \\
&y.(r_0\otimes r_1\otimes \cdots \otimes r_n)=(-1)^{\frac{1}{2}n(n+1)}\bar{r}_0\otimes \bar{r}_n \otimes \bar{r}_{n-1}\otimes \cdots \otimes \bar{r}_1.
\end{align*}
Let $\bD_n(\cR)$ denote the space of coinvariants $(\cR^{\otimes n+1})_{D_{n+1}}$. If we modify the action of $y$ by $-y$, the resulting coinvariants will be denoted by ${}_{-1}\bD_n(\cR)$.
The Hochschild boundary $b:\cR^{\otimes n+1} \rightarrow \cR^{\otimes n}$ is given by
\begin{align*}
b(r_0\otimes r_1 \otimes \cdots \otimes r_n)=
&\sum_{i=0}^{n-1}(-1)^{i}r_0 \otimes \cdots \otimes r_ir_{i+1} \otimes \cdots \otimes r_n 
\\
&+(-1)^n r_n r_0 \otimes r_1 \otimes \cdots \otimes r_{n-1}.
\end{align*}
It is compatible with passing to the quotient by the action of the dihedral group (cf. \cite{L}). Denote the obtained complices by $(\bD_\bullet(\cR),\bar{b})$
(resp. $({}_{-1}\bD_\bullet(\cR),\bar{b})$). Their homologies are called \emph{dihedral }(resp. \emph{skew-dihedral}) \emph{homology} of $\cR$:
\[ HD_n(\cR):=H_n(\bD_\bullet(\cR),\bar{b}), \qquad 
  (\; \text{resp}.\; {}_{-1}HD_n(\cR):=H_n({}_{-1}\bD_\bullet(\cR),\bar{b})\; ).
\]
\begin{rem} \label{rem_cyclic-dihedral}
It is known $($cf. \cite{L} $)$ that
\begin{align*}
HD_n(\cR)&=HC_n^+(\cR):=\{ m \in HC_n(\cR) \, \vert \, y.m=m \}, \\
{}_{-1}HD_n(\cR)&=HC_n^-(\cR):=\{ m \in HC_n(\cR) \, \vert \, y.m=-m \},
\end{align*} 
since our field $k$ contains $\frac{1}{2}$. Here, $HC_\bullet(\cR)$ is the cyclic homology. 
\end{rem}
\section{Feigin-Tsygan Type Isomorphisms}
First recall an isomorphism of Lie algebras $\fg J(R) \overset{\sim}{\rightarrow} \fgl_n(J(R))$
due to B. Feigin and B. Tsygan (Lemma 1.1 in \cite{FI}). 
Restricting such isomorphisms to $\fo_J^{\odd}(R), \fsp_J(R)$ and $\fo_J^{\even}(R)$, we obtain an interesting description of these algebras with an appropriate anti-involution of $J(R)$. 
\subsection{Isomorphisms between $\fg J(R)$ and $\fgl_n(J(R))$ ($n>1$)}
Fix an integer $n>1$. Let $I \subset \Z$ be the set of representatives of $\Z/\vert I\vert \Z$. Let $\Phi_I: \fg J(R) \rightarrow \fgl_{\vert I\vert}(J(R))$ defined by
\[ \Phi_I(E_{m+r\vert I\vert, n+s\vert I\vert})=e_{m,n}(E_{r,s}) \qquad m,n \in I \; \text{and}\; r,s \in \Z,
\]
where $e_{m,n}(E_{r,s})$ is the matrix whose only non-zero entry is $E_{r,s} \in J(R)$ at the $(m,n)$-entry.
It can be checked that the maps $\Phi_I$ are isomorphisms of Lie algebras.

Typically, we choose for $n \in \Z_{>0}$,  $I_{2n+1}:=\{i \in \Z \vert -n\leq i\leq n\}$ and $I_{2n}:=\{i \in \Z \vert -n \leq i<n\}$. \\

In the rest of this section, we analyze the restriction of these isomorphisms to the appropriate Lie subalgebras. 
\subsection{The Choice of Real Forms of $\fo_n(R)$ and $\fsp_{2n}(R)$}\label{sect_def-classic}
Here, inspired by \cite{LP}, we define $\fo_n(R)$ and $\fsp_{2n}(R)$ over an associative unital $k$-algebra $R$ with a $k$-linear anti-involution $\overline{\cdot}:R \rightarrow R$. From now on, we denote the matrix elements of finite size matrices by $e_{i,j}$ and set $e_{i,j}(r)=re_{i,j}$ for $r \in R$.\\
\vskip+0.1in

\noindent \fbox{$\fo_{2n+1}(R)$} \hspace{0.2in} This Lie algebra is the Lie subalgebra of \\ \phantom{ABCDEFGHI}  $\fgl_{I_{2n+1}}(R)= \bigoplus_{-n\leq i,j\leq n} R e_{i,j}$ 
defined by
\[ \fo_{2n+1}(R)=\{\,X \in \fgl_{I_{2n+1}}(R)\, \vert \, {}^t(X)J_n^{B}+J_n^{B}X=0\}, \]
where $J_n^{B}=\sum_{i=-n}^n e_{i,-i}$. Note that $(J_n^{B})^2=I=\sum_{i=-n}^n e_{i,i}$. 

\vskip+0.15in

\noindent \fbox{$\fsp_{2n}(R)$} \hspace{0.2in} This Lie algebra is the Lie subalgebra of \\ \phantom{ABCDEFGh}  $\fgl_{I_{2n}}(R)=\bigoplus_{-n\leq i,j< n} R e_{i,j}$ 
defined by
\[ \fsp_{2n}(R)=\{\,X \in \fgl_{I_{2n}}(R)\, \vert \, {}^t(X)J_n^{C}+J_n^{C}X=0\}, \]
where $J_n^{C}=\sum_{i=-n}^{n-1} (-1)^{i}e_{i,-i-1}$. Note that $(J_n^{C})^2=-I=-\sum_{i=-n}^{n-1} e_{i,i}$.\\

\vskip+0.15in

\noindent \fbox{$\fo_{2n}(R)$} \hspace{0.2in} This Lie algebra is the Lie subalgebra of \\
\phantom{ABCDEFG} $\fgl_{I_{2n}}(R)=\bigoplus_{-n\leq i,j<n} R e_{r,s}$ defined by
\[ \fo_{2n}(R)=\{\,X \in \fgl_{I_{2n}}(R)\, \vert \, {}^t(X)J_{n}^{D}+J_{n}^{D}X=0\}, \]
where $J_{n}^{D}=\sum_{i=-n}^{n-1}e_{i,-i-1}$. Note that $(J_{n}^{D})^2=I=\sum_{i=-n}^{n-1} e_{i,i}$. \\

\vskip+0.15in

Notice that the natural inclusions $I_{2n+1} \hookrightarrow I_{2n+3}$ and $I_{2n} \hookrightarrow I_{2n+2}$ induce natural inclusions $\fo_{2n+1}(R) \hookrightarrow \fo_{2n+3}(R)$, (resp. $\fo_{2n}(R) \hookrightarrow \fo_{2n+2}(R)$ and $\fsp_{2n}(R) \hookrightarrow \fsp_{2n+2}(R)$) which allows us to define their stable limits $\fo_{\odd}(R)$ (resp. $\fo_{\even}(R)$ and $\fsp(R)$). 
\begin{rem} Since there are natural inclusions $\fo_{2n} \hookrightarrow \fo_{2n+1}$ and $\fo_{2n+1} \hookrightarrow \fo_{2n+2}$, the two inductive limits $\fo_{\odd}(R)$ and $\fo_{\even}(R)$ are in fact the same. Nevertheless, for the clarity, we will keep these separate notations.
\end{rem}

Let $\fgl(R)$ be the Lie algebra of the matrices $M=(m_{i,j})$ with coefficients in $R$ and indexed over $\Z$, such that the set $\{(i,j) \in \Z^2\, \vert \, m_{i,j}\neq 0\, \}$ is finite. 
\begin{rem} The Lie algebras $\fo_{\odd}(R), \fsp(R)$ and $\fo_{\even}(R)$ can be defined as subalgebras of $\fgl(R)$ as follows.  
Set 
\[ J_B=\sum_{i \in \Z} e_{i,-i}, \qquad J_C=\sum_{i \in \Z}(-1)^{i}e_{i,-i-1}, \qquad
J_D=\sum_{i \in \Z}e_{i,-i-1}. \]
Define the anti-involutions of $\fgl(R)$ by
\[ \tau_B(X)=J_B{}^t(X)J_B, \qquad \tau_C(X)=-J_C{}^t(X)J_C, \qquad \tau_D=J_D{}^t(X)J_D. \]
$($cf. See Remark \ref{rem_matrix-J}.$)$
It is clear that 
\[ \fo_{\odd}(R)\cong \fgl(R)^{\tau_B,-}, \qquad 
   \fsp(R)\cong \fgl(R)^{\tau_C,-}, \qquad
   \fo_{\even}(R)\cong \fgl(R)^{\tau_D,-}.
\]
\end{rem}
\subsection{Restrictions of $\Phi_I$}
Let $\ast$ be the anti-involution on $J(R)$ satisfying
\[ E_{k,l}(r)^{\ast}=E_{-l,-k}(\bar{r}) \qquad r \in R. \]
\begin{prop}\label{prop_restriction}
Let $n \in \Z_{>1}$. The following restrictions are isomorphisms:
\begin{enumerate}
\item 
$\Phi_{I_{2n+1}}:\fo_J^{\odd}(R) \overset{\sim}{\longrightarrow} (\fo_{2n+1}(J(R)), \ast)$
\item 
$\Phi_{I_{2n}}:\fsp_J(R) \overset{\sim}{\longrightarrow}  (\fsp_{2n}(J(R)), \ast),$
\item 
$\Phi_{I_{2n}}:\fo_J^{\even}(R) \overset{\sim}{\longrightarrow} (\fo_{2n}(J(R)), \ast).$
\end{enumerate}
\end{prop}
\begin{proof} Since the proofs are similar, we just show the first case. 
Let $r \in R$. For $i,j \in I_{2n+1}$ and $k,l \in \Z$, we have 
\begin{align*}
&\Phi(E_{i+(2n+1)k,j+(2n+1)l}(r)-E_{-j-(2n+1)l,-i-(2n+1)k}(\bar{r})) \\
=
&(e_{i,j}(E_{k,l}(r))-e_{-j,-i}(E_{-l,-k}(\bar{r})). 
\end{align*} 
\end{proof}
In particular, Proposition \ref{prop_restriction} implies
\begin{cor} 
\begin{enumerate}\label{cor-stable-lim}
\item $H_\bullet(\fo_J^{\odd}(R)) \cong H_\bullet(\fo_{\odd}(J(R)))$,
\item $H_\bullet(\fsp_J(R)) \cong H_\bullet(\fsp(J(R)))$,
\item$H_\bullet(\fo_J^{\even}(R)) \cong H_\bullet(\fo_{\even}(J(R)))$,
\end{enumerate}
\end{cor}

\section{Homology of $\fo_J^{\odd}(R), \,\fsp_J(R)$ and $\fo_J^{\even}(R)$}
Let $k$ be a field of characteristic $0$ and $R$ an associative unital $k$-algebra. 
In this section, we determine the homology of the Lie algebras $\fo_J^{\odd}(R), \fsp_J(R)$ and $\fo_J^{\even}(R)$. As a corollary, we obtain an explicit realization 
of their universal central extension. 

\subsection{Primitive Part of the Homology}
For a Hopf $k$-algebra $H$, we denote its primitive part by $\Prim(H)$. 
Thanks to Corollary \ref{cor-stable-lim},
we can directly apply Theorem 5.5 of \cite{LP}
and we obtain the next result:
\begin{thm}\label{thm_main0} Let $\ast:J(R) \rightarrow J(R)$ be the anti-involution satisfying
\[ E_{k,l}(r)^{\ast}=E_{-l,-k}(\bar{r}), \qquad r \in R. \]
Then for $\fg=\fo_J^{\odd}, \,\fsp_J$ and $\fo_J^{\even}$, we have
\[
\Prim(H_\bullet(\fg(R)))
=
{}_{-1}HD_{\bullet-1}(J(R)).
\]
\end{thm} 
In the next subsection, we explain how to apply Theorem 5.5 of \cite{LP} in our situation.

\subsection{Sketch of the proof}
We restrict ourselves to the orthogonal case, since the proof for $\fsp_{2n}$ is the same as in \cite{LP}. 

Let $I$ be at most countable (index) set and let $\cF$ be the free algebra over $\{x_i, x_i^{\ast}\}_{i \in I}$. Denote the set of monomials in $\cF$ by $\cM$. The algebra $\cF$ admits the anti-involution $\ast$ defined by $\ast(x_i)=x_i^{\ast}$ ($i \in I$). 
Let $\sim$ be the relation on $\cM$ defined by
\begin{enumerate}
\item (anti-involution) $m \sim m^{\ast}$, 
\item (cyclic equivalence) $m_1m_2 \sim m_2 m_1$.
\end{enumerate}
It can be checked that $\sim$ is an equivalence relation. For $m \in \cM$, denote its equivalence class in $\cM/\sim$ by $\Tr(m)$. Let us denote the polynomial ring in the variables in $\cM/\sim$ by $\cT$. The multiplicative group $D=(k^{\ast})^{I}$ acts on $\cF$ induced from the action on $\bigoplus_{i \in I}(kx_i \oplus kx_i^{\ast})$ given by
\[ d.\left(\sum_{i \in I} ( c_i x_i +c_i^{\ast}x_i^{\ast})\right)=\sum_{i \in I} d_i ( c_i x_i +c_i^{\ast}x_i^{\ast}), \]
where $c_i, c_i^{\ast} \in k$ and $d=(d_i)_{i \in I} \in D$. This $D$-action induces a $D$-action on $\cT$. An element $f \in \cT$ is called \emph{multilinear} if for every $d \in D$, we have $f^d=\det(d)f$. For a finite $I$, the set of multilinear elements forms a finite dimensional vector space, denoted by $\cT_{\vert I\vert}$. 

Recall that the \emph{hyperoctahedral group} $H_ n$ is the semi-direct product $(\Z/2\Z)^n\rtimes \fS_n$, where the symmetric group $\fS_n$ acts on $(\Z/2\Z)^n$ by permuting the factors. The group $H_n$ acts on $\cT$ as follows: $\fS_n$ acts by the same permutations on $x_i$'s and $x_i^{\ast}$'s,  the $i$th generator $\eta_i$ of $\Z/2\Z$ permutes $x_i$ and $x_i^{\ast}$ and  leaves the other variables unchanged. So $H_n$ is the subgroup of $\fS_{2n}$ of the group of permutations on $\{x_1, \ldots, x_n, x_1^{\ast}, \ldots, x_n^{\ast}\}$ which commutes with the operator $\ast$. 

A key step we need in invariant theory for the orthogonal group is the $H_n$-isomorphism:
for $p\geq n$, 
\[ \Ind_{H_n}^{\fS_{2n}}(\triv) \; \overset{\sim}{\longleftarrow} \; \cT_n
   \; \overset{\sim}{\longrightarrow} \; [(\fgl_p^{\otimes n})_{O_p}]^{\ast},
\]
where $\triv$ is the trivial representation of $H_n$. See, e.g., \cite{P} and \cite{W} for details.
\begin{rem} The isomorphisms in the above diagram remain valid for any choice of $O_p$ as subgroup of $GL_p$. 
\end{rem}

Now we explain the steps of the proof of Theorem \ref{thm_main0}. For simplicity, set $A=J(R),  
\tau_\odd=\tau_B$ and $\tau_{\even}=\tau_D$. \\

\noindent \textit{First step} \hspace{0.1in} 
Let us define $\fo_\sharp(A)=(\fgl(A))_{\Z/2\Z}$ where $\Z/2\Z$ acts by $\alpha \mapsto \alpha^{\tau_\sharp}$ ($\sharp \in \{\even, \odd\}$).
Hence, 
\begin{align*}
\bigwedge{}^n \fo_\sharp(A)=
&(\fo_\sharp(A)^{\otimes n})_{\fS_n}=((\fgl(A)^{\otimes n})_{(\Z/2\Z)^n})_{\fS_n} \\
=
&(\fgl(A)^{\otimes n})_{H_n}=(\fgl^{\otimes n}\otimes A^{\otimes n})_{H_n},
\end{align*}
where $\eta_i$ acts on $\alpha=\alpha_1\otimes \cdots \otimes \alpha_n \in \fgl^{\otimes n}$ by 
$\eta_i(\alpha)=\alpha_1 \otimes \cdots \otimes \alpha_i^{\tau_\sharp} \otimes \cdots \otimes \alpha_n$, and on $a=a_1\otimes \cdots \otimes a_n \in A^{\otimes n}$ 
by $\eta_i(a)=a_1 \otimes \cdots \otimes (-a_i^{\ast})\otimes \cdots \otimes \alpha_n$.
A permutation $\sigma \in \fS_n$ acts on $\fgl^{\otimes n}$ by permuting the variables, and on $A^{\otimes n}$ by permuting the variables and multiplying by $\sgn(\sigma)$. \\

\noindent \textit{Second step} \hspace{0.1in}By Proposition 6.4 of \cite{LQ},  the 
Chevalley-Eilenberg complex $(\bigwedge^{\bullet} \fo_p(A),d)$ (cf. \S \ref{sect_Lie-homology})  is quasi-isomorphic to the complex of coinvariants $((\bigwedge^{\bullet} \fo_p(A))_{O_p},d)$.
The last module can be written as
\[ (\bigwedge{}^n \fo_\sharp(A))_{O_\sharp}=((\fgl^{\otimes n}\otimes A^{\otimes n})_{H_n})_{O_\sharp}=
   ((\fgl^{\otimes n})_{O_\sharp}\otimes A^{\otimes n})_{H_n},
\]
where $O_{\sharp}$ is the ind-algebraic group associated to the ind Lie-algebra $\fo_\sharp$.
Using the above mentioned invariant theory result, we get
\[ (\bigwedge{}^n \fo_\sharp(A))_{O_\sharp} \cong (\Ind_{H_n}^{\fS_{2n}}(\triv)\otimes A^{\otimes n})_{H_n}. 
\]

\noindent \textit{Third step}
\hspace{0.1in} We regard $\fS_{2n}$ as the group of permutations on $\{1,2,\ldots, n, \\1^{\ast}, 2^{\ast}, \ldots, n^{\ast}\}$ where the subgroups $\fS_n \subset H_n$ act as simultaneous permutations on $\{1,2,\ldots, n\}$ and $\{1^{\ast}, 2^{\ast}, \ldots, n^{\ast}\}$ and $\eta_i \in (\Z/2\Z)^n$ permutes $i$ and $i^{\ast}$ and leaves the other elements unchanged. Via the isomorphism $\cT_n
   \; \overset{\sim}{\longrightarrow} \; [(\fgl_p^{\otimes n})_{O_p}]^{\ast}$, the primitive elements of $(\bigwedge^n \fo_\sharp(A))_{O_\sharp}$ correspond to linear combinations of single traces, i.e., the elements of the form $\Tr(y_{i_1}y_{i_2}\cdots y_{i_n})$ with $y_i \in \{x_i, x_i^{\ast}\}$ and $\{i_1,i_2,\ldots, i_n\}=\{1,2,\ldots, n\}$
(cf. the coproduct formula in the proof of Proposition 6.6 in \cite{LQ}). Notice that such trace element corresponds to 
\[ (\prod_{i \, \text{s.t.}\, y_i=x_i^{\ast}}\eta_i) \\(i_1,i_2,\cdots, i_n). 
\]
The $k$-span of such elements forms an $H_n$-submodule of $\Ind_{H_n}^{\fS_{2n}}(\triv)$ generated by $\kappa \otimes 1$ where $\kappa =(1,2,\ldots, n) \in \fS_{2n}$.

It turns out that $\mathrm{Stab}_{H_n}(\kappa \otimes 1)$ is isomorphic to the dihedral group $D_n$. Indeed, the stabilizer is generated by $\kappa_H=\kappa \kappa^{\ast}$ and $\eta \omega_H$ with
\begin{align*}
& \kappa^{\ast}=(1^{\ast}, 2^{\ast}, \ldots, n^{\ast}), \qquad \eta=\prod_{i=1}^n \eta_i, \\
&\omega_H=\begin{pmatrix} 1 & 2 & \cdots & n \\ n & n-1 & \cdots & 1 \end{pmatrix}\begin{pmatrix}
1^{\ast} & 2^{\ast} & \cdots & n^{\ast} \\ n^{\ast} & (n-1)^{\ast} & \cdots & 1^{\ast} \end{pmatrix}.
\end{align*}
Thus we obtain
\[ \mathrm{Prim}((\bigwedge{}^{\bullet} \fo_{\sharp} (A))_{O_{\sharp}})=(\Ind_{D_n}^{H_n}(\triv) \otimes A^{\otimes n})_{H_n}=(A^{\otimes n})_{D_n}, \]
where $\triv$ here denotes the trivial representation of $D_n$.

Set $x=\kappa_H$ and $y=\kappa_H \eta \omega_H$. By direct computation, 
\begin{align*}
&x.(a_1\otimes a_2\otimes \cdots \otimes a_n)=
(-1)^{n-1}(a_n \otimes a_1\otimes \cdots \otimes a_{n-1}), \\
&y.(a_1\otimes a_2\otimes \cdots \otimes a_n)=
(-1)^{\frac{1}{2}(n+1)(n+2)}(a_1^{\ast}\otimes a_n^{\ast} \otimes a_{n-1}^{\ast}\otimes \cdots \otimes a_2^{\ast}).
\end{align*}
With an isomorphism $\mathrm{Prim}((\bigwedge{}^{\bullet} \fo_{\sharp} (A))_{O_{\sharp}}) \cong (A^{\otimes n})_{D_n}$, 
the restriction of the differential of the complex $((\bigwedge{}^{\bullet} \fo_{\sharp} (A))_{O_{\sharp}}, d)$ to its primitive part corresponds to the Hochschild boundary operator of the dihedral complex $({}_{-1}\bD_\bullet(A),\bar{b})$. \\

In the rest of this section, we relate these skew-dihedral homologies of $J(R)$ with (skew-)dihedral homologies of $R$. For this purpose, we recall that the (skew-)dihedral homology is, by definition, obtained by a particular $\Z/2\Z$-action on the Hochschild complex (hence, the Connes bicomplex etc).
\subsection{$\Z/2\Z$-actions on Hochschild Homologies}
In \cite{FI}, we have proved that the isomorphism between the Hochschild homologies $\Phi_p: HH_p(R) \rightarrow HH_{p+1}(J(R))$ is induced from the morphism of the abelian groups $\widetilde{\Phi}_p:R^{\otimes p+1} \rightarrow J(R)^{\otimes p+2}$ defined by
\begin{align*}
& \widetilde{\Phi}_p(r_0\otimes r_1\otimes \cdots \otimes r_p) \\
=
&r_0 I\otimes \left( \sum_{l=0}^p (-1)^l r_1 I\otimes \cdots \otimes r_l I\otimes N \otimes r_{l+1}I \otimes \cdots \otimes r_p I\right),
\end{align*}
where $N=\sum_{i \in \Z} e_{i,i+1}$ (cf. see Remark \ref{rem_N}). 

We recall from \cite{L} that the $\Z/2\Z=\langle y_p^R \rangle$-action on $R^{\otimes p+1}$ is given by 
\[ y_p^R.(r_0\otimes r_1\otimes \cdots \otimes r_p)=(-1)^{\frac{1}{2}p(p+1)}(\bar{r}_0\otimes \bar{r}_p \otimes \bar{r}_{p-1}\otimes \cdots \otimes \bar{r}_1), \]
and the $\Z/2\Z=\langle y_{p+1}^{J(R)} \rangle$-action on $J(R)^{\otimes p+2}$ is given by
\[ y_{p+1}^{J(R)}.(m_0\otimes m_1\otimes \cdots \otimes m_{p+1})=(-1)^{\frac{1}{2}(p+1)(p+2)}(m_0^{\ast}\otimes m_{p+1}^{\ast}\otimes m_p^{\ast} \otimes \cdots \otimes m_1^{\ast}). \]

Let us analyze the commutativity of these $\Z/2\Z$-actions with the map $\Phi_p$
(or rather $\widetilde{\Phi}_p$). By definition, we have
\begin{align*}
&\widetilde{\Phi}_p (y_p^R.(r_0\otimes r_1\otimes \cdots \otimes r_p))
=
(-1)^{\frac{1}{2}p(p+1)}\widetilde{\Phi}_p(\bar{r}_0\otimes \bar{r}_p \otimes \bar{r}_{p-1}\otimes \cdots \otimes \bar{r}_1) \\
=
&(-1)^{\frac{1}{2}p(p+1)}
\bar{r}_0I\otimes \left(\sum_{l=0}^p(-1)^l\bar{r}_p I\otimes \cdots \bar{r}_{p+1-l}I\otimes N \otimes \bar{r}_{p-l}I\otimes \cdots \otimes \bar{r}_1I\right),
\end{align*}
and 
\begin{align*}
&y_{p+1}^{J(R)}.\widetilde{\Phi}_p(r_0\otimes r_1\otimes \cdots \otimes r_p) \\
=
&y_{p+1}^{J(R)}.\left[ r_0I \otimes \left(\sum_{l=0}^p(-1)^lr_1I\otimes \cdots r_lI \otimes N \otimes r_{l+1}I \otimes \cdots \otimes r_pI\right)\right] \\
=
&(-1)^{\frac{1}{2}(p+1)(p+2)}
\bar{r}_0I \otimes \left(\sum_{l=0}^p(-1)^l \bar{r}_pI\otimes \cdots \otimes \bar{r}_{l+1}I\otimes N^{\ast}\otimes \bar{r}_lI\otimes \cdots \otimes \bar{r}_1I\right) \\
=
&-(-1)^{\frac{1}{2}p(p+1)}
\bar{r}_0I \otimes \left(\sum_{l=0}^p(-1)^l \bar{r}_pI\otimes \cdots \otimes \bar{r}_{p-l+1}I\otimes N^{\ast}\otimes \bar{r}_{p-l}I\otimes \cdots \otimes \bar{r}_1I\right).
\end{align*}
By Theorem \ref{thm_main0} and the definition of the anti-involution $\ast$, we have $N^{\ast}= N$, from which we conclude
\begin{lemma} For $p \in \Z_{\geq 0}$, we have
 $y_{p+1}^{J(R)}\circ \Phi_p=-\Phi_p \circ y_p^R$.
\end{lemma}
In particular, by Remark \ref{rem_cyclic-dihedral}, this implies ${}_{-1}HD_{\bullet-1}(J(R))=HD_{\bullet-2}(R)$ (cf. \S \ref{sect_dihedral}).

\subsection{Main result}
We obtain
\begin{thm}\label{thm_main} For $\fg=\fo_J^{\odd}, \, \fsp_J$ and $\fo_J^{\even}$, we have
\[
\Prim(H_\bullet(\fg(R)))
= 
HD_{\bullet-2}(R).
\]
\end{thm}
In particular, it is known that for $R=k$, $HD_i(k)=k$ for $i \equiv 0 \mod 4$,  and $HD_i(k)=0$ otherwise $($cf. Example 1.9 in \cite{L}$)$. Hence, 
\begin{cor} For $\fg=\fo_J^{\odd}, \, \fsp_J$ and $\fo_J^{\even}$, $\Prim(H_\bullet(\fg(k)))$ is the graded $k$-vector space whose $i$th graded component $\Prim(H_\bullet(\fg(k)))_i$ is 
\[\Prim(H_\bullet(\fg(k)))_i =\begin{cases}\; k \; & i \equiv 2 \, (4), \\
                                                                  \; 0 \; & \text{otherwise}. \end{cases}
\]
\end{cor}
\begin{rem} These numbers are exactly the double of the ``exponents of the Weyl groups'' $($cf. \cite{B}$)$ of type $B_J, C_J$ and $D_J$, respectively. A similar phenomenon happens also for $\fg J(k)$, see Theorem 6.6 in \cite{FI}. 
\end{rem}

\subsection{Universal Central Extension}\label{sect_UCE}

Theorem \ref{thm_main} implies
\begin{cor} For $\fg=\fo_J^{\odd}, \, \fsp_J$ and $\fo_J^{\even}$, we have
\[ H_2(\fg (R))=HD_0(R)=(R^{ab})_1, \]
where $(R^{ab})_1$ is the fixed point part of the anti-involution $\overline{\cdot}$ defined on $R^{ab}=R/[R,R]$, that is induced from the anti-involution $\overline{\cdot}$ on $R$. 
\end{cor}
Thus, the kernel of the universal central extension $\widetilde{\fg}(R)$ of $\fg(R)$ is $(R^{ab})_1$, whose Lie bracket is given by the restriction of those of $\widetilde{\fg J}(R)$ described in \S 6.4 of \cite{FI} as follows.

Set $I_+=\sum_{i\geq 0} e_{i,i}$ and $I_-=\sum_{i<0} e_{i,i}$. It is clear that $I_{\pm} \in J(R)$ and $I_\sigma I_\tau=\delta_{\sigma, \tau} I_\tau$ for $\sigma, \tau \in \{\pm\}$. Let $\Phi:J(R) \rightarrow J(R)$ be the $k$-linear map defined by $\Phi(X)=I_+XI_+$.  A matrix $M=(m_{i,j}) \in J(R)$ is said to be of finite support, if the set $\{(i,j)\, \vert m_{i,j} \neq 0\}$ is finite. We denote the Lie subalgebra of $\fg J(R)$ consisting of the matrices of finite support by $\fg F(R)$.  Define the trace map $\Tr: \fg F(R) \rightarrow R^{ab}$ as the composition of the usual trace map $\mathrm{tr}: \fg F(R) \rightarrow R; M=(m_{i,j})\, \mapsto \sum_i m_{i,i}$ and the abelianization $\pi^{ab}: R \twoheadrightarrow R^{ab}$. 

Let $\Psi: J(R) \times J(R) \rightarrow R^{ab}$ be the $k$-bilinear map defined by
\begin{align*}
\Psi(X,Y)=
&\Tr([\Phi(X), \Phi(Y)]-\Phi([X,Y])) \\
=
&\Tr((I_+YI_-)(I_-XI_+)-(I_+XI_-)(I_-YI_+)).
\end{align*}
It can be checked that this is a $2$-cocycle.
\begin{thm} For $\fg=\fo_J^{\odd}, \, \fsp_J$ and $\fo_J^{\even}$,
the Lie algebra $\widetilde{\fg}(R)$ is a $k$-vector space
$\widetilde{\fg}(R)=\fg(R) \oplus (R^{ab})_1$
equipped with the Lie bracket $[ \cdot, \cdot ]'$ defined by 
\[
[\widetilde{\fg}(R), (R^{ab})_1]'=0, \qquad
[X,Y]'=[X,Y]+\Psi(X,Y) \quad X,Y \in \fg(R).
\]
\end{thm}

\newpage


\appendix
\section{Free field realization of $\fg J(\C)$}
In this Appendix, we explain the historical origin of the so-called Japanese cocycle (cf. \cite{JM}). 
\subsection{Clifford algebras}
Let $\cA$ be the Clifford algebra over $\C$ generated by $\psi_i, \psi_i^{\ast} \; (i \in \Z)$ and $1$,  subject to the relations
\[ [\psi_i, \psi_j]_+=0, \qquad [\psi_i, \psi_j^{\ast}]_+=\delta_{i,j}1, \qquad
[\psi_i^{\ast},\psi_j^{\ast}]_+=0. \]
An element of $\cW=(\bigoplus_{i \in \Z}\C \psi_i)\oplus (\bigoplus_{i \in \Z}\C \psi_i^{\ast})$ is referred to as a {\sl free fermion}.  Set
\[ \cW_{\an}:=\left(\bigoplus_{i<0} \C\psi_i\right)\oplus \left(\bigoplus_{i\geq 0}\C \psi_i^{\ast}\right), \qquad
   \cW_{\cre}:=\left(\bigoplus_{i\geq 0} \C\psi_i\right)\oplus \left(\bigoplus_{i< 0}\C \psi_i^{\ast}\right),
\]
($\an$  stands for annihilation and $\cre$ stands for creation) $\cF:=\cA/\cA\cW_{\an}$ and $\cF^{\ast}:=\cW_{\cre}\cA \slash \cA$. It is clear that $\cF$ has a left $\cA$-module structure and $\cF^{\ast}$ has a right $\cA$-module structure. 
We denote the image of $1 \in \cA$ in $\cF$ and $\cF^{\ast}$ by $\vacr$ and $\vacl$ respectively. $\cF$ and $\cF^{\ast}$ are the $\cA$-modules generated by 
$\vacr$ and $\vacl$ with the defining relations:
\[ \cW_{\an}.\vacr=0, \qquad \vacl.\cW_{\cre}=0. \]

The image of  $\fgl(\C) \hookrightarrow \End_\C(\cF)$ (and $\fgl(\C)\hookrightarrow \End_\C(\cF^{\ast})$) can be described as follows:
\[ \{ \,\sum_{\text{finite}}a_{i,j}\psi_i\psi_j^{\ast}\, \}.
\]

\subsection{Japanese cocycle}

The Lie algebra $\fg J(\C)$ does not act on $\cF$ and $\cF^{\ast}$.  Indeed, it is the universal central extension $\widetilde{\fg J}(\C)$ of $\fg J(\C)$ (cf. \cite{FI}) that can be viewed as a subalgebra of $\End_\C(\cF)$ (resp. $\End_\C(\cF^{\ast})$). To describe this explicitly, let us introduce the so-called \emph{normal ordered product} $\nod \cdot \nod$ as follows (for detail, see \cite{JM}).

By PBW theorem for $\cA$, it follows that $\cA=(\cW_{\cre}\cA+\cA\cW_{\an})\oplus \C$. We denote the canonical projection $\cA \twoheadrightarrow \C$ with respect to this decomposition by $\pi$. By definition, the $\C$-bilinear map $\langle \cdot \rangle: \cF^{\ast}\times \cF \longrightarrow \C; (\vacl.a, b.\vacr) \longmapsto \pi(ab)$ is well-defined for any $a,b \in \cA$. Explicitly, some of them are given by
\begin{align*}
& \langle \psi_i\psi_j\rangle:=0, \qquad \langle \psi_i^{\ast}\psi_j^{\ast}\rangle:=0, \\
& \langle \psi_i\psi_j^{\ast}\rangle:=
    \begin{cases} \; \delta_{i,j} \; &i=j<0, \\
                           \; 0 \; & \text{otherwise} \end{cases} \qquad
    \langle \psi_j^{\ast} \psi_i\rangle:=
    \begin{cases} \; \delta_{i,j}\; &i=j\geq 0, \\
                          \; 0 \; & \text{otherwise} \end{cases}.
\end{align*}
We set $\nod \psi_i \psi_j^{\ast} \nod := \psi_i\psi_j^{\ast} -\langle \psi_i\psi_j^{\ast}\rangle$. The subset of $\End_\C(\cF)$ (resp. $\End_\C(\cF^{\ast})$) defined by 
\[ \{ \, \sum_{i,j} a_{i,j} \nod \psi_i\psi_j^{\ast}\nod \,
 \vert \, \exists\,\, N \,\, \text{s.t.} \,\, a_{i,j}=0 \quad (\text{$\forall\, i,j$} \,\, \text{s.t.} \,\, \vert i-j\vert>N)\, \} \oplus \C \cdot 1
\]
is closed under the Lie bracket. Explicitly, this is given by 
\begin{align*}
& [\sum_{i,j} a_{i,j}\nod \psi_i\psi_j^{\ast}\nod, \sum_{k,l} b_{k,l} \nod \psi_k\psi_l^{\ast}\nod]\\
=
&\sum_{i,j}\left(\sum_k(a_{i,k}b_{k,j}-b_{i,k}a_{k,j})\right)\nod \psi_i\psi_j^{\ast}\nod+\left(\sum_{i<0,j\geq 0}a_{i,j}b_{j,i}-\sum_{i\geq 0, j<0}a_{i,j}b_{j,i}\right) 1. 
\end{align*}
This second term is the so-called ``\textbf{Japanese cocycle}''.

\section{Lie Algebras $\fo_J^{\odd}(\C), \fsp_J(\C)$ and $\fo_J^{\even}(\C)$}
\label{sect_BCD-JM}
In this section, we briefly recall how the orthogonal and symplectic 
subalgebras of $\fg J(\C)$ were introduced. 
\subsection{Some involutions on $J(\C)$}\label{sect_involution-I}
After recalling the involutions $\sigma_l$ of $\cA$ given in \cite{JM},
we compute the induced (anti-)involutions on $J(\C)$. \\

For $l \in \Z$, define the involution $\sigma_l$ of $\cA$ by 
\[ \sigma_l(\psi_n)=(-1)^{l-n}\psi_{l-n}^{\ast}, \qquad 
   \sigma_l(\psi_n^{\ast})=(-1)^{l-n}\psi_n. 
\]
By definition, one has 
\[ \sigma_l(\psi_i\psi_j^{\ast})=(-1)^{i+j}\psi_{l-i}^{\ast}\psi_{l-j}
=(-1)^{i+j}(-\psi_{l-j}\psi_{l-i}^{\ast}+\delta_{i,j}1), \]
i.e., the involution $\sigma_l$ induces an involution $\widetilde{\sigma}_l$ of $\fg J(R)$ given by $\widetilde{\sigma}_l(E_{i,j})=-(-1)^{i+j}E_{l-j,l-i}$. Indeed, it should be regarded as $(-1)$ times the anti-involution $\tau_l$ on $J(R)$ defined by
$\tau_l(E_{r,s})=(-1)^{r+s}E_{l-s,l-r}$. Remark, that by setting $J_l=\sum_{i \in \Z}(-1)^{i}E_{i,l-i}$,  we get $(J_l)^2=(-1)^l I$ with $I=\sum_{i \in \Z}E_{i,i}$. An element $X \in J(R)$ is $\langle -{\tau}_l\rangle$-invariant iff
\[ {}^tX+(-1)^{l}J_lXJ_l={}^tX+J_lXJ_l^{-1}=0,  \]
since $J_l E_{r,s}J_l=(-1)^{r+s+l}E_{l-r,l-s}$  $($cf. see Remark \ref{rem_matrix-J}$)$. 

Now, for an associative unital $\C$-algebra $R$ with an anti-involution $\overline{\cdot}:R \rightarrow R$, we extend the definition of the \textbf{transpose}, also denoted by ${}^t(\cdot )$, to $J(R)$ by ${}^t(rE_{i,j})=\overline{r}E_{j,i}$. Set
\[ \tau_l(X)=(-1)^lJ_l{}^t(X)J_l \qquad X \in J(R). \]
This is a well-defined anti-involution on $J(R)$ that induces an anti-involution on $\fg J(R)$. 

\begin{rem}\label{rem_N} Set $N=\sum_{i \in \Z} e_{i,i+1}$. 
\begin{enumerate}
\item ${}^tNN=I$, i.e., $N^{-1}={}^tN$, 
\item $N^{-1}\tau_l(X)N=\tau_{l+2}(\Ad(N)(X))$ for $l \in \Z$ and $X \in J(R)$ \\
since $N^{-1}J_lN=-J_{l+2}$.
\end{enumerate}
Hence, it is sufficient to consider $\tau_0$ and $\tau_{-1}$, for example.
\end{rem}

\subsection{Fermions with $2$ components}\label{sect_involution-II}
We may think of the fermions $\psi_n^{(j)}, \psi_n^{(j) \ast}$ indexed by $n \in \Z$ and $j \in \{0,1\}$ satisfying
\[ [\psi_m^{(j)}, \psi_n^{(k)}]_+=0, \qquad [\psi_m^{(j)},\psi_n^{(k)\ast}]_+=\delta_{j,k}\delta_{m,n}1, \qquad [\psi_m^{(j)\ast}, \psi_n^{(k)\ast}]_+=0.
\]
Fixing a bijection $\Z \times \{0,1\} \rightarrow \Z$, the Clifford algebra generated by fermions with $2$ components is isomorphic to $\cA$. For example, one of the simplest choices is
\[ \psi_m^{(0)}=\psi_{2m}, \quad \psi_m^{(1)}=\psi_{2m+1}, \qquad
   \psi_m^{(0)\ast}=\psi_{2m}^{\ast}, \quad \psi_m^{(1)\ast}=\psi_{2m+1}^{\ast}. \]
Here and after, we fix such a renumeration, if necessary.

Let $\sigma$ be the involution of these fermions with $2$ components, defined by
\[ \sigma(\psi_n^{(j)})=(-1)^n\psi_{-n}^{(j)\ast}, \qquad \sigma(\psi_n^{(j)\ast})=(-1)^n \psi_{-n}^{(j)}. \]
By definition, one has
\[ \sigma(\psi_m^{(j)}\psi_n^{(k)\ast})=(-1)^{m+n}\psi_{-m}^{(j)\ast}\psi_{-n}^{(k)}
=(-1)^{m+n}(-\psi_{-n}^{(k)}\psi_{-m}^{(j)\ast}+\delta_{j,k}\delta_{m,n}1).
\]
Enumerating matrix elements of $J(k)$ as $E_{m,n}^{(j),(k)}$, the involution $\sigma$ induces the involution $\widetilde{\sigma}$ of $\fg J(\C)$ given by
$\widetilde{\sigma}(E_{m,n}^{(j),(k)})=-(-1)^{m+n}E_{-n,-m}^{(k),(j)}$. Indeed, it should be regarded as $(-1)$ times the anti-involution $\tau$ on $J(k)$ defined by
$\tau(E_{m,n}^{(j),(k)})=(-1)^{m+n}E_{-n,-m}^{(k),(j)}$. Set 
\[J=\sum_{j=0}^1\sum_{m \in \Z}(-1)^mE_{m,-m}^{(j),(j)}. \]
It is clear that $J^2=I=\sum_{j=0}^1\sum_{m \in \Z}E_{m,m}^{(j),(j)}$ and $X \in J(\C)$ is $\langle -\tau \rangle$-invariant iff
\[ {}^tX+JXJ=0, \]
since $JE_{m,n}^{(j),(k)}J=(-1)^{m+n}E_{-m,-n}^{(j),(k)}$ $($cf. see Remark \ref{rem_matrix-J}$)$.

Now, on $J(R)$ (hence on $\fg J(R)$), the anti-involution $\tau$ is defined by
\[ \tau(X)=J{}^t(X)J, \]
as in the previous subsection. 
\subsection{Central extension of $\fo_J^{\odd}(\C), \fsp_J(\C)$ and $\fo_J^{\even}(\C)$}
Recall that all of these Lie algebras are defined as the fix point subalgebras with respect to certain involutions. So their central extensions may also be realized as the fix point subalgebras with respect to certain involutions (cf. \cite{JM}). Explicitly, Lie algebras over $\C$ of type $B_J$, $C_J$ and $D_J$ are defined by $\widetilde{\fg J}(\C)^{\tau_0,-}, \widetilde{\fg J}(\C)^{\tau_{-1},-}$ and $\widetilde{\fg J}(\C)^{\tau,-}$, respectively. 

Our main theorem assures that they are, in fact,  the universal central extensions of the Lie algebras $\fo_J^{\odd}(\C), \fsp_J(\C)$ and $\fo_J^{\even}(\C)$, respectively.

\end{document}